\documentclass[12pt,oneside,english]{amsart}
\usepackage[T1]{fontenc}
\usepackage[latin9]{inputenc}
\usepackage{amsthm}
\usepackage{amstext}
\usepackage{amssymb}
\usepackage{esint}

\makeatletter
\numberwithin{equation}{section}
\numberwithin{figure}{section}
\theoremstyle{plain}
\newtheorem{thm}{Theorem}
  \theoremstyle{plain}
  \newtheorem{prop}[thm]{Proposition}
  \theoremstyle{plain}
  \newtheorem{cor}[thm]{Corollary}
  \theoremstyle{plain}
  \newtheorem*{thm*}{Theorem}
  \theoremstyle{remark}
  \newtheorem{rem}[thm]{Remark}

\makeatother

\usepackage{babel}

\begin{document}

\title{Minimality, (Weighted) Interpolation in Paley-Wiener Spaces \& Control
Theory}

\author{Fr\'ed\'eric Gaunard}
\begin{abstract}
It is well known from a result by Shapiro-Shields that in the Hardy
spaces, a sequence of reproducing kernels is uniformly minimal if
and only if it is an unconditional basis in its span. This property
which can be reformulated in terms of interpolation and so-called
weak interpolation is not true in Paley-Wiener spaces in general.
Here we show that the Carleson condition on a sequence $\Lambda$
together with minimality in Paley-Wiener spaces $PW_{\tau}^{p}$ of
the associated sequence of reproducing kernels implies the interpolation
property of $\Lambda$ in $PW_{\tau+\epsilon}^{p}$, for every $\epsilon>0$.
With the same technics, using a result of McPhail, we prove a similary
result about minimlity and weighted interpolation in $PW_{\tau+\epsilon}^{p}$..
We apply the results to control theory, establishing that, under some
hypotheses, a certain weak type of controllability in time $\tau>0$
implies exact controllability in time $\tau+\epsilon$, for every
$\epsilon>0$.
\end{abstract}

\subjclass[2000]{30E05, 42A15, 93B05.}

\keywords{Interpolation, Paley-Wiener spaces, minimal sequences, controllability.}

\date{\today}

\maketitle

\section{Introduction}

Let $X$ be a Banach space. A sequence $\left\{ \phi_{n}\right\} _{n\geq1}$
of vectors of $X$ is said to be \emph{minimal }in $X$ if $\phi_{n}\not\in\bigvee_{k\neq n}\phi_{k}:=\overline{\text{span}}^{X}\left(\phi_{k}:\: k\not=n\right)$,
$n\geq1$, and \emph{uniformly minimal} if moreover \begin{equation}
\inf_{n\geq1}\text{dist}\left(\frac{\phi_{n}}{\left\Vert \phi_{n}\right\Vert },\:\bigvee_{k\neq n}\phi_{k}\right)>0.\label{def: unif min.}\end{equation}
It is well known (see e.g. \cite[p. 93]{Ni02a}) that minimality of
$\left\{ \phi_{n}\right\} _{n\geq1}$ in $X$ is equivalent to the
existence of a sequence $\left\{ \psi_{n}\right\} _{n\geq1}\subset X^{\star}$
such that $\left\langle \phi_{n},\psi_{k}\right\rangle =\delta_{nk}$
and the minimality is said uniform if and only if \begin{equation}
\sup_{n\geq1}\:\left\Vert \phi_{n}\right\Vert \cdot\left\Vert \psi_{n}\right\Vert <\infty.\label{eq: unif min equiv}\end{equation}

We consider the case where $X$ is a Banach space of analytic functions
on a domain $\Omega$. Let $\Lambda=\left\{ \lambda_{n}\right\} _{n\geq1}$
be a sequence of complex numbers lying in $\Omega$. We use the terminology
\emph{minimal }also for the sequence $\Lambda$ if there exists a
sequence of functions $(f_{n})_{n\geq1}$ of $X$ such that\[
f_{n}(\lambda_{k})=\delta_{nk},\quad n,k\geq1,\]
and we say that $\Lambda$ is a \emph{weak interpolating sequence}
in $X$, which is denoted by $\Lambda\in\text{Int}_{w}\left(X\right)$,
if there exists a sequence of functions $\left(f_{n}\right)_{n\geq1}$
of $X$ such that\begin{equation}
f_{n}\left(\lambda_{k}\right)=\delta_{nk}\left\Vert k_{\lambda_{n}}\right\Vert _{X^{\star}},\; n\geq1,\text{ and }\sup_{n\geq1}\left\Vert f_{n}\right\Vert <\infty.\label{Int w}\end{equation}
When $X$ is reflexive, this is equivalent to the fact that $\mathcal{K}(\Lambda)$
is uniformly minimal in $X^{\star}$. Such a sequence $\Lambda$ could
also be called a uniformly minimal sequence in $X$ but we prefer
to keep the existing terminology of weak interpolating sequence.

In the case where $X=H^{p}\left(\mathbb{C}_{a}^{\pm}\right)$, $1<p<\infty$,
the Hardy space of the half-plane $\mathbb{C}_{a}^{\pm}$, we can
identify \[
\left(H^{p}\left(\mathbb{C}_{a}^{\pm}\right)\right)^{\star}\simeq H^{q}\left(\mathbb{C}_{a}^{\pm}\right),\qquad\frac{1}{p}+\frac{1}{q}=1,\]
and it is known that the reproducing kernel at $\lambda\in\mathbb{C}_{a}^{\pm}$
is given by $k_{\lambda_{n}}\left(z\right)=\frac{i}{2\pi}\left(z-\overline{\lambda}_{n}\right)^{-1}$.
We have the estimate \[
\left\Vert k_{\lambda_{n}}\right\Vert _{H^{q}\left(\mathbb{C}_{a}^{\pm}\right)}\asymp\left|\text{Im}\left(\lambda_{n}\right)-a\right|^{-\frac{1}{p}}.\]
From factorization in Hardy spaces, it can be deduced that the condition
(\ref{Int w}) is equivalent to the so-called\emph{ Carleson condition}
\begin{equation}
\text{inf}_{n\geq1}\prod_{k\not=n}\left|\frac{\lambda_{n}-\lambda_{k}}{\lambda_{n}-\overline{\lambda}_{k}-2ia}\right|>0,\label{Carleson}\end{equation}
and by the above observations, this is equivalent to $\mathcal{K}(\Lambda)$
being uniformly minimal in $H^{q}\left(\mathbb{C}_{a}^{\pm}\right)$.

In this paper, the spaces $X$ that we consider will only be the Hardy
spaces $H^{p}\left(\mathbb{C}_{a}^{\pm}\right)$ or the Paley-Wiener
spaces $PW_{\tau}^{p}$ to be defined later.

We say that $\Lambda$ is an \emph{interpolating sequence} for $X=H^{p}\left(\mathbb{C}_{a}^{\pm}\right)$
or $PW_{\tau}^{p}$, which is denoted by $\Lambda\in\text{Int}\left(X\right)$,
if for each sequence $a=\left(a_{n}\right)_{n\geq1}\in l^{p}$, there
is a fonction $f\in X$ such that \begin{equation}
f\left(\lambda_{n}\right)=a_{n}\left\Vert k_{\lambda_{n}}\right\Vert _{X^{\star}},\quad n\geq1,\label{int def}\end{equation}
and a \emph{complete interpolating sequence} for $X$ ($\Lambda\in\text{Int}_{c}\left(X\right)$)
if the function satisfying (\ref{int def}) is unique. We will give
the explicit formula of $k_{\lambda_{n}}$ for $PW_{\tau}^{p}$ and
an estimate of $\left\Vert k_{\lambda_{n}}\right\Vert _{PW_{\tau}^{q}}$
in the next section. 

A famous result by Shapiro and Shields (\cite{SS61}) states that
in $H^{p}\left(\mathbb{C}_{a}^{\pm}\right)$, the Carleson condition
(\ref{Carleson}) for a sequence $\Lambda$ is equivalent to the interpolation
property of $\Lambda$. It is also known that $\Lambda\in\text{Int}\left(H^{p}\left(\mathbb{C}_{a}^{\pm}\right)\right)$
if and only if $\mathcal{K}\left(\Lambda\right)$ is an unconditional
basis (or, for $p=2$, a Riesz basis) in its span in $H^{q}\left(\mathbb{C}_{a}^{\pm}\right)$.
We refer to \cite[Section C, Chapter 3]{Ni02b} for definitions and
details. It appears that in the Hardy spaces, the uniformly minimal
sequences are exactly the unconditional sequences.

This property of equivalence between uniform minimality and unconditionnality
is not isolated. It turns out to be true in the Bergmann space (\cite{ScS98}),
in the Fock spaces and in the Paley Wiener spaces for certain values
of $p$ (\cite{ScS00}).

In \cite{AH10}, the authors show that uniform minimality implies
unconditionality in a bigger space for certain backward shift invariant
spaces $K_{I}^{p}:=H^{p}\cap I\overline{H_{0}^{p}}$ (considered here
on the unit circle $\mathbb{T}$) for which the Paley-Wiener spaces
are a particular case. We will use here a different approach to obtain
a stronger result. More precisely, considering the unit disk, the
authors of that paper increase the size of the space $K_{I}^{p}$
in two directions: $K_{J}^{s}$, where $s<p$ and $J$ is an inner
multiple of $I$. In our situation of the Paley-Wiener space $PW_{\tau}^{p}$,
which is isometric to $K_{I_{\mathbb{D}}^{\tau}}^{p}$, $I_{\mathbb{D}}^{\tau}(z)=\exp\left(2\tau(z+1)/(z-1)\right)$
for $z\in\mathbb{D}$, we still increase the size of the space by
taking an inner multiple of $I_{\mathbb{D}}^{\tau}$ but we keep the
same $p$.

We have already mentioned that $\Lambda$ is an interpolating sequence
for the Hardy space $H^{p}$, $1<p<\infty$, if and only if the sequence
$\mathcal{K}(\Lambda)$ is an unconditional basis in its span in $H^{q}$
(see e.g. \cite{Ni02b} or \cite{Se04}). Hence, the result of Shapiro
and Sheilds implies that weak interpolation is equivalent to interpolation
in Hardy spaces. A characterization of complete interpolating sequences
in $PW_{\tau}^{p}$ obtained by Lyubarskii and Seip (\cite{LS97})
(involing in particular Carleson's condition and the Muckenhoupt $(A_{p})$
condition on the generating function of $\Lambda$) implies that Paley-Wiener
spaces do not have this property.

Indeed, as shown in \cite{ScS00}, the sequence $\Lambda=\left\{ \lambda_{n}:\: n\in\mathbb{Z}\right\} $
defined by\[
\lambda_{0}=0,\quad\lambda_{n}=n+\frac{\text{sign}\left(n\right)}{2\max\left(p,q\right)},\; n\in\mathbb{Z}\setminus\left\{ 0\right\} ,\]
is weak interpolating in $PW_{\pi}^{p}$ and complete, but does not
satisfy the conditions of Lyubarskii-Seip's result, and so, $\Lambda$
is not a (complete) interpolating sequence in $PW_{\pi}^{p}$. Nevertheless,
as we will discuss in Subsection \ref{sub:Upper-Uniform-Density},
a density argument (see \cite{Se95}) allows to show that this sequence
is actually an interpolating sequence in a bigger space, i.e. in $PW_{\pi+\epsilon}^{p}$,
for arbitrary $\epsilon>0$. This is a special case of the main result
of this paper.
\begin{thm}
\label{thm: MAIN RESULT}Let $\tau>0$, $1<p<\infty$ and $\Lambda$
be a minimal sequence in $PW_{\tau}^{p}$ such that for every $a\in\mathbb{R}$,
$\Lambda\cap\mathbb{C}_{a}^{\pm}$ satisfies the Carleson condition
(\ref{Carleson}). Then, for every $\epsilon>0$, $\Lambda$ is an
interpolating sequence in $PW_{\tau+\epsilon}^{p}$.
\end{thm}
It should be emphasized that surprisingly, we do not need to require
uniform minimality here. The Carleson condition allows in a way to
compensate this lack of uniformity. As a consequence of this result,
we will see that if $\Lambda\in\text{Int}_{w}\left(PW_{\tau}^{p}\right)$,
then, for every $\epsilon>0$, $\Lambda\in\text{Int}\left(PW_{\tau+\epsilon}^{p}\right)$. 

Finally, we recall that a positive measure $\sigma$ on some half-plane
$\mathbb{C}_{a}^{\pm}$ is called a \emph{Carleson measure} in $\mathbb{C}_{a}^{\pm}$
if\begin{equation}
\sup_{Q}\frac{\sigma\left(Q\right)}{h}<\infty,\label{eq:Carleson measure}\end{equation}
where the supremum is taken over all the squares $Q$ of the form\[
Q=\left\{ z=x+iy\in\mathbb{C}_{a}^{\pm}:\: x_{0}<x<x_{0}+h,\;\left|y-a\right|<h\right\} ,\]
for $x_{0}\in\mathbb{R}$ and $h>0$. It is well known from a result
of Carleson (see e.g. \cite[pp. 61 and 278]{Ga81}) that $(1)\Rightarrow(2)\Leftrightarrow(3)$,
where

$(1)$ The sequence $\Lambda=\left\{ \lambda_{n}:\: n\geq1\right\} \subset\mathbb{C}_{a}^{\pm}$
satisfies the Carleson condition (\ref{Carleson});

$(2)$ The measure\[
\sigma_{\Lambda}:=\sum_{n\geq1}\left|\text{Im}\left(\lambda_{n}\right)-a\right|\delta_{\lambda_{n}}\]
is a Carleson measure in $\mathbb{C}_{a}^{\pm}$;

$(3)$ For every $f\in H^{p}\left(\mathbb{C}_{a}^{\pm}\right)$, \[
\int\left|f\right|^{p}d\sigma_{\Lambda}\lesssim\left\Vert f\right\Vert _{p}^{p}.\]
It is also known that $(2)$ or $(3)$ together with the \emph{uniform
pseudo-hyperbolic separation} of $\Lambda$ in $\mathbb{C}_{a}^{\pm}$,
which is \begin{equation}
\inf_{n\neq m}\left|\frac{\lambda_{n}-\lambda_{m}}{\lambda_{n}-\overline{\lambda}_{m}-2ia}\right|>0,\label{unif psh separation}\end{equation}
imply $(1)$. Moreover, if $\Lambda$ lies in a strip of finite width,
$i.e.$ $M:=\sup_{n}\left|\text{Im}\left(\lambda_{n}\right)\right|<\infty$,
the Carleson condition (\ref{Carleson}) in $\mathbb{C}_{\mp2M}^{\pm}$
is equivalent to the \emph{uniform separation condition} \begin{equation}
\inf_{n\neq m}\left|\lambda_{n}-\lambda_{m}\right|>0\label{uniform separation}\end{equation}
which is, in this case, equivalent to the uniform pseudo-hyperbolic
separation since the pseudo-hyperbolic metric defined in $\mathbb{C}_{\mp2M}^{\pm}$
by\[
\rho\left(\lambda,\mu\right)=\left|\frac{\lambda-\mu}{\lambda-\overline{\mu}-2iM}\right|\]
is locally equivalent to the euclidian distance.

This paper is organized as follows. The next section will be devoted
to interpolation in Paley-Wiener spaces. After having recalled some
properties of these spaces, we discuss links between density and interpolation
(in the case of the sequence $\Lambda$ lying in a strip of finite
width) and prove our main result and some consequences. 

In the third section, we define and discuss weighted interpolation.
Indeed, after having defined weighted interpolation in Hardy and Paley-Wiener
spaces, we use a result of McPhail (\cite{McP90}) characterizing
the weighted interpolation sequences in $H^{p}\left(\mathbb{C}_{a}^{\pm}\right)$
and technics of Theorem \ref{thm: MAIN RESULT} to prove that a minimal
sequence in $PW_{\tau}^{p}$ such that its intersection with every
half-plane satisfies the McPhail condition is a weighted interpolation
sequence in $PW_{\tau+\epsilon}^{p}$, for every $\epsilon>0$.

This theorem will be used in the fourth and last section where we
consider controllability of linear differential systems, establishing
a link between exact and a certain weak type of controllability.

\section{Interpolation in Paley-Wiener Spaces}

We begin by recalling some facts about Paley-Wiener spaces. For $\tau>0$,
the Paley-Wiener space $PW_{\tau}^{p}$ consists of all entire functions
of exponential type at most $\tau$ satisfying\[
\left\Vert f\right\Vert _{p}^{p}=\int_{-\infty}^{+\infty}\left|f(t)\right|^{p}dt<\infty.\]
A result, known as Plancherel-Polya inequality (see e.g. \cite{Le96}
or \cite[p.95]{Se04}) states that if $f\in PW_{\tau}^{p}$, then\begin{equation}
\int_{-\infty}^{+\infty}\left|f(x+iy)\right|^{p}dx\leq e^{p\tau|y|}\left\Vert f\right\Vert _{p}^{p}.\label{eq:plancherel polya}\end{equation}
In particular, it follows that for every $f\in PW_{\tau}^{p}$, $z\mapsto e^{\pm i\tau z}f(z)$
belongs to $H^{p}\left(\mathbb{C}_{a}^{\pm}\right)$, with same norm
as $f$. It also follows that translation is an isomorphism from $PW_{\tau}^{p}$
into itself. 

Using Cauchy's formula and Cauchy's theorem in an appropriate way,
we see that \[
k_{\lambda}(z)=\frac{\sin\tau\left(z-\overline{\lambda}\right)}{\tau\left(z-\overline{\lambda}\right)}\]
is the reproducing kernel of $PW_{\tau}^{p}$ associated to $\lambda$
and we obtain the norm estimate\[
\left\Vert k_{\lambda}\right\Vert _{PW_{\tau}^{q}}\asymp\left(1+\left|\text{Im}(\lambda)\right|\right)^{-\frac{1}{p}}e^{\tau\left|\text{Im}(\lambda)\right|}.\]
This implies a useful pointwise estimate; recalling that for $\frac{1}{p}+\frac{1}{q}=1$,
$\left(PW_{\tau}^{p}\right)^{\star}\simeq PW_{\tau}^{q}$, we deduce
that there exists a constant $C=C(p)$ such that for every $f\in PW_{\tau}^{p}$,
we have\begin{equation}
\left|f(z)\right|\leq C\left\Vert f\right\Vert _{p}\left(1+\left|\text{Im}\left(z\right)\right|\right)^{-\frac{1}{p}}e^{\tau\left|\text{Im}(z)\right|},\quad z\in\mathbb{C}.\label{pointwise est. PW}\end{equation}
The Paley-Wiener theorem states that\[
L^{2}(0,2\tau)\simeq\mathcal{F}L^{2}(-\tau,\tau)=PW_{\tau}^{2},\]
where $\mathcal{F}$ denotes the Fourier transform\[
\mathcal{F}\phi(z)=\int_{-\tau}^{\tau}\phi(t)e^{-itz}dt.\]
 Hence, another approach to interpolation problems in $PW_{\tau}^{2}$
is to consider geometric properties of exponentials in $L^{2}(0,2\tau)$,
which is a famous problem with several applications, see e.g. \cite{HNP81}
or \cite{AI95}.

From the definitions given previously, a sequence $\Lambda$ is interpolating
in $PW_{\tau}^{p}$ if, for every sequence $a=\left(a_{n}\right)_{n\geq1}\in l^{p}$,
it is possible to find a function $f\in PW_{\tau}^{p}$ such that\begin{equation}
f\left(\lambda_{n}\right)\left(1+\left|\text{Im}\left(\lambda_{n}\right)\right|\right)^{\frac{1}{p}}e^{-\tau\left|\text{Im}\left(\lambda_{n}\right)\right|}=a_{n},\quad n\geq1.\label{eq:def int PW}\end{equation}

The condition of weak interpolation for $\Lambda$ in $PW_{\tau}^{p}$
can be reformulated as the existence of a sequence of functions $\left(f_{n}\right)_{n\geq1}\subset PW_{\tau}^{p}$
such that\[
f_{n}\left(\lambda_{k}\right)\left(1+\left|\text{Im}\left(\lambda_{n}\right)\right|\right)^{\frac{1}{p}}e^{-\tau\left|\text{Im}\left(\lambda_{n}\right)\right|}=\delta_{nk},\quad n\geq1,\]
and $\sup_{n\geq1}\left\Vert f_{n}\right\Vert <\infty$.

In particular, if $\Lambda\in\text{Int}_{w}\left(PW_{\tau}^{p}\right)$,
then the Plancherel-Polya inequality (\ref{eq:plancherel polya})
implies that the sequence $\left(e^{\pm i\tau\cdot}f_{n}\right)_{n\geq1}$is
in $H^{p}\left(\mathbb{C}_{a}^{\pm}\right)$, with uniform control
of the norm. So, it is easy to see that $\Lambda\cap\mathbb{C}_{a}^{\pm}\in\text{Int}_{w}\left(H^{p}\left(\mathbb{C}_{a\pm\eta}^{\pm}\right)\right)$,
for every $\eta>0$, and hence satisfies the Carleson condition in
the corresponding half-plane, in view of Shapiro-Shields 's theorem.
Moreover, we can affirm that the sequence \[
\Lambda_{a,\eta}:=\Lambda\cap\left\{ z\in\mathbb{C}:\;0<\left|\text{Im}(z)-a\right|<2\eta\right\} \]
is uniformly separated, in view of the discussion in the end of the
previous section. These two observations imply the following result
(for more details, see \cite{Gau11}).
\begin{prop}
\label{w Int implies Carleson}If $\Lambda$ is a weak interpolating
sequence in $PW_{\tau}^{p}$, then, for every $a\in\mathbb{R}$, the
sequence $\Lambda\cap\mathbb{C}_{a}^{\pm}$ satisfies the Carleson
condition (\ref{Carleson}).
\end{prop}

\subsection{\label{sub:Upper-Uniform-Density}Upper Uniform Density and Interpolation}

In this subsection, we assume that the sequence $\Lambda$ satisfies\begin{equation}
M:=\sup_{n\geq1}\left|\text{Im}\left(\lambda_{n}\right)\right|<\infty,\label{lambda incluse dans bande}\end{equation}
which means that $\Lambda$ lies in a strip of finite width parallel
to the real axis. We define the \emph{upper uniform density} $\mathcal{D}_{\Lambda}^{+}$
by\[
\mathcal{D}_{\Lambda}^{+}:=\lim_{r\to\infty}\frac{n_{\Lambda}^{+}\left(r\right)}{r},\]
where \[
n_{\Lambda}^{+}\left(r\right):=\sup_{x\in\mathbb{R}}\left|\text{Re}\left(\Lambda\right)\cap\left[x,x+r\right]\right|,\]
counting multiplicities. 

The reader would have remembered that, from previous remarks, Proposition
\ref{w Int implies Carleson} implies that a weak interpolating sequence
in a Paley-Wiener space $PW_{\tau}^{p}$ satisfies the uniform separation
condition.

The next theorem is stated as follows in a paper of Seip (\cite[Theorem 2.2]{Se95})
the proof of which is based on a more general result by Beurling (\cite{Be89}).
\begin{thm}
(\cite{Se95})\label{thm:(densite et interpolation)} Let $\Lambda$
be a sequence satisfying (\ref{lambda incluse dans bande}).

If $\Lambda$ is uniformly separated and $\mathcal{D}_{\Lambda}^{+}<\frac{\tau}{\pi}$
, then $\Lambda\in\text{Int}\left(PW_{\tau}^{p}\right)$. Conversely,
if $\Lambda\in\text{Int}\left(PW_{\tau}^{p}\right)$, then, $\Lambda$
is necessarily uniformly separated and $\mathcal{D}_{\Lambda}^{+}\leq\frac{\tau}{\pi}$.\end{thm}
\begin{cor}
\label{cor:densite n+sg/2p'}The sequence $\Lambda=\left\{ \lambda_{n}:\: n\in\mathbb{Z}\right\} $
defined by \[
\lambda_{0}=0\text{ et }\lambda_{n}=n+\frac{\text{sign}(n)}{2\max\left(p,q\right)},\quad n\neq0,\]
is interpolating in $PW_{\pi+\epsilon}^{p}$, for every $\epsilon>0$.\end{cor}
\begin{proof}
We have already mentioned that this sequence is uniformly minimal.
The uniform separation condition is obvious. Its upper uniform density
is clearly equal to $1$. The Corollary now follows from Theorem \ref{thm:(densite et interpolation)}.
\end{proof}

\subsection{Proof of Main Result}

We recall our main theorem, previously stated in the first section.
\begin{thm*}
(Theorem \ref{thm: MAIN RESULT}) Let $\tau>0$, $1<p<\infty$ and
$\Lambda$ be a minimal sequence in $PW_{\tau}^{p}$ such that for
every $a\in\mathbb{R}$, $\Lambda\cap\mathbb{C}_{a}^{\pm}$ satisfies
the Carleson condition (\ref{Carleson}). Then, for every $\epsilon>0$,
$\Lambda$ is an interpolating sequence in $PW_{\tau+\epsilon}^{p}$.\end{thm*}
\begin{proof}
Using an idea of Beurling, let $\epsilon>0$ be fixed and $\phi_{\epsilon}\in\mathcal{C}^{\infty}$,
with compact support contained in $\left(-\frac{\epsilon}{2},\frac{\epsilon}{2}\right)$,
be such that $H_{\epsilon}:=c\mathcal{F}\phi_{\epsilon}$ satisfies
$H_{\epsilon}(0)=1$. In particular, the Paley-Wiener theorem implies
that $H_{\epsilon}$ is an entire function of exponential type $\epsilon$.
Moreover, since $\phi_{\epsilon}$ belongs to the Schwarz class (and
in particular is $\mathcal{C}^{1}$), we have t\[
\left|H_{\epsilon}\left(x\right)\right|\lesssim\frac{1}{1+|x|},\quad x\in\mathbb{R}.\]
Now, from a a Phragmen-Lindelf principle (see e.g. \cite[p.39]{Le96}),
we can deduce that \begin{equation}
\left|H_{\epsilon}\left(z\right)\right|\lesssim\frac{e^{\epsilon\left|\text{Im}\left(z\right)\right|}}{1+\left|z\right|},\quad z\in\mathbb{C}.\label{majoration Hepsilon}\end{equation}

On the other hand, since $\Lambda$ is minimal in $PW_{\tau}^{p}$,
there exists a sequence of functions $\left(f_{\lambda}\right)_{\lambda\in\Lambda}\subset PW_{\tau}^{p}$
such that $f_{n}(\lambda_{k})=\delta_{nk}$. Let $a=\left(a_{n}\right)_{n\geq1}$
be a finitely supported sequence and $f$ be the solution of the interpolation
problem \[
f(\lambda_{n})=a_{n},\quad n\geq1,\]
defined by\[
f\left(z\right)=\sum_{n\geq1}a_{n}f_{n}\left(z\right)H_{\epsilon}\left(z-\lambda_{n}\right).\]
(Notice that $f$ is a finite sum of functions belonging to $PW_{\tau+\epsilon}^{p}$.)
From (\ref{eq:def int PW}), it suffices to bound the quantity\[
\inf\left\{ \left\Vert f-g\right\Vert _{p}:\; g\in PW_{\tau+\epsilon}^{p},\; g(\lambda)=0,\;\lambda\in\Lambda\right\} \]
by a constant times the following norm of $a$ \[
\left\Vert a\right\Vert :=\left(\sum_{\lambda\in\Lambda}\left|a_{\lambda}\right|^{p}\left(1+\left|\text{Im}\left(\lambda\right)\right|\right)e^{-p\left(\tau+\epsilon\right)\left|\text{Im}\left(\lambda\right)\right|}\right)^{\frac{1}{p}}.\]
We split the above sum in two parts: $f^{+}$ and $f^{-}$ corresponding
respectively to $\Lambda_{0}^{+}:=\Lambda\cap\left(\mathbb{C}^{+}\cup\mathbb{R}\right)$
and $\Lambda\cap\mathbb{C}^{-}$, and estimate each part separately.
Here, we will only estimate the first one, the method is the same
for the second one. Let $\eta>0$ be such that $\left\{ \text{Im}\left(z\right)=-\eta\right\} \cap\Lambda=\emptyset$.
The Plancherel-Polya inequality allows us to estimate the norm of
$f^{+}-g^{+}$, $g^{+}\in PW_{\tau+\epsilon}^{p}$ and $g^{+}|\Lambda=0$,
on the axis $\left\{ \text{Im}\left(z\right)=-\eta\right\} $. We
consider the Blaschke product associated to $\Lambda_{-\eta}^{+}$,
in the half-plane $\mathbb{C}_{-\eta}^{+}$\[
B_{-\eta}\left(z\right)=\prod_{\lambda_{n}\in\Lambda_{-\eta}^{+}}c_{\lambda_{n}}\frac{z-\lambda_{n}}{z-\overline{\lambda}_{n}-2i\eta},\quad z\in\mathbb{C}_{-\eta}^{+},\]
with suitable unimodular coefficients $c_{\lambda_{n}}$. For $\lambda_{n}\in\Lambda_{0}^{+}$,
we consider the function\[
G_{\lambda_{n},\epsilon}:z\mapsto\left(z-\lambda_{n}\right)H_{\epsilon}\left(z-\lambda_{n}\right)f_{n}\left(z\right)e^{i\left(\tau+\epsilon\right)z}\]
which belongs to $H^{\infty}\left(\mathbb{C}_{-\eta}^{+}\right)$
(this follows from (\ref{eq:plancherel polya}), (\ref{pointwise est. PW})
and (\ref{majoration Hepsilon})) and vanishes on $\Lambda_{-\eta}^{+}$
(it actually vanishes on $\Lambda$). We recall that $\Lambda_{-\eta}^{+}$
satisfies the Carleson condition in $\mathbb{C}_{-\eta}^{+}$. Also,
the function $G_{\lambda_{n},\epsilon}^{0}:=G_{\lambda_{n},\epsilon}/B_{-\eta}$
belongs to $H^{\infty}\left(\mathbb{C}_{-\eta}^{+}\right)$. Let $B^{-}$
be the Blaschke product associated to $\Lambda_{-\eta}^{-}$ (in $\mathbb{C}_{-\eta}^{-}$).
Observe that\[
\inf\left\{ \left\Vert f^{+}-g^{+}\right\Vert _{p}:\; g^{+}\in PW_{\tau+\epsilon}^{p},\; g^{+}|\Lambda=0\right\} \]
\[
=\inf\left\{ \left\Vert \sum_{\lambda_{n}\in\Lambda^{+}}a_{n}\frac{G_{\lambda_{n},\epsilon}^{0}}{z-\lambda_{n}}-g_{0}^{+}\right\Vert _{p}:\; g_{0}^{+}\in Y\right\} \]
with\[
g_{0}^{+}=\frac{g^{+}}{B_{-\eta}}e^{i\left(\tau+\epsilon\right)\cdot}\]
and \[
Y:=H_{+}^{p}\cap\overline{B_{-\eta}}\left(K_{I^{\tau+\epsilon}}^{p}\cap I^{\tau+\epsilon}B^{-}H_{-}^{p}\right)\subset L^{p}\left(\mathbb{R}\right).\]
By duality arguments inspired by Shapiro and Shields (see \cite[p. 516]{SS61}
and \cite{Gau11} where we consider the bilinear form $\left(f,g\right):=\int_{\mathbb{R}-i\eta}fg$,
for $f,g\in L^{p}\left(\mathbb{R}-i\eta\right)$), and because\[
Y^{\bot_{\left(\cdot,\cdot\right)}}=H_{+}^{q}+Z\]
where $Z$ is such that, for every $h\in Z$, we have \[
\int_{\mathbb{R}-i\eta}\left(\sum_{\lambda_{n}\in\Lambda^{+}}a_{n}\frac{G_{\lambda_{n},\epsilon}^{0}}{z-\lambda_{n}}\right)hdm=0,\]
it is enough to estimate\[
\sup\left\{ N(h):\; h\in H^{q}\left(\mathbb{C}_{-\eta}^{+}\right),\left\Vert h\right\Vert =1\right\} ,\]
where \begin{eqnarray*}
N\left(h\right) & := & \left|\int_{\mathbb{R}-i\eta}\left(\sum_{\lambda_{n}\in\Lambda_{0}^{+}}a_{n}\frac{G_{\lambda_{n},\epsilon}^{0}}{z-\lambda_{n}}\right)hdm\right|.\\
 & \:= & \left|\sum_{\lambda_{n}\in\Lambda_{0}^{+}}a_{n}\int_{\mathbb{R}}\frac{G_{\lambda_{n},\epsilon}^{0}\left(x-i\eta\right)h\left(x-i\eta\right)}{x-\left(\lambda_{n}+i\eta\right)}dx\right|.\end{eqnarray*}
Now, $z\mapsto G_{\lambda_{n},\epsilon}^{0}\left(z-i\eta\right)h\left(z-i\eta\right)$
is a function in $H_{+}^{q}$ and the Cauchy formula gives\[
N(h)=\left|\sum_{\lambda_{n}\in\Lambda_{0}^{+}}a_{n}G_{\lambda_{n},\epsilon}^{0}\left(\lambda_{n}+i\eta-i\eta\right)h\left(\lambda_{n}+i\eta-i\eta\right)\right|.\]
Moreover, since $\Lambda_{0}^{+}$ satisfies the Carleson condition
in $\mathbb{C}_{-\eta}^{+}$ , we have $\left|\frac{B_{-\eta}}{b_{\lambda_{n}}}\left(\lambda_{n}\right)\right|\asymp1$
and since $f_{\lambda_{n}}\left(\lambda_{n}\right)H_{\epsilon}\left(0\right)=1$,
we can estimate \[
\left|G_{\lambda_{n},\epsilon}^{0}\left(\lambda_{n}\right)\right|\asymp\left(\text{Im}\left(\lambda_{n}\right)+\eta\right)e^{-\left(\tau+\epsilon\right)\text{Im}\left(\lambda_{n}\right)},\quad\lambda_{n}\in\Lambda_{0}^{+}.\]
It follows from the triangle inequality and Hlder's inequality that\begin{eqnarray*}
N(h) & \lesssim & \left(\sum_{\lambda_{n}\in\Lambda_{0}^{+}}\left|a_{n}\right|^{p}\left(1+\text{Im}(\lambda_{n})\right)e^{-p\left(\tau+\epsilon\right)\text{Im}(\lambda_{n})}\right)^{\frac{1}{p}}\\
 &  & \times\left(\sum_{\lambda_{n}\in\Lambda_{0}^{+}}\text{Im}\left(\lambda_{n}+i\eta\right)\left|\tilde{h}\left(\lambda_{n}+i\eta\right)\right|^{q}\right)^{\frac{1}{q}},\end{eqnarray*}
where $\tilde{h}=h\left(\cdot-i\eta\right)\in H_{+}^{q}$ . Now, the
Carleson condition satisfied by $\Lambda_{0}^{+}+i\eta$ in $\mathbb{C}^{+}$
gives \[
\left(\sum_{\lambda_{n}\in\Lambda_{0}^{+}}\text{Im}\left(\lambda_{n}+i\eta\right)\left|\tilde{h}\left(\lambda_{n}+i\eta\right)\right|^{q}\right)^{\frac{1}{q}}\lesssim\left\Vert h\right\Vert =1.\]
See (\ref{eq:Carleson measure}) and properties mentioned thereafter.
Finally, we obtain\[
\inf\left\{ \left\Vert f^{+}-g\right\Vert :\; g\in PW_{\tau+\epsilon}^{p},\; g|\Lambda=0\right\} \]
\[
\qquad\qquad\lesssim\left(\sum_{\lambda_{n}\in\Lambda_{0}^{+}}\left|a_{n}\right|^{p}\left(1+\text{Im}(\lambda_{n})\right)e^{-p\left(\tau+\epsilon\right)\text{Im}(\lambda_{n})}\right)^{\frac{1}{p}}\]
which is the required estimate and ends the proof.
\end{proof}
To conclude this section, we give two immediate corollaries to our
main theorem. First, since, by Proposition \ref{w Int implies Carleson},
 a weak interpolating sequence in $PW_{\tau}^{p}$ has to satisfy
the Carleson condition in every half-plane $\mathbb{C}_{a}^{\pm}$,
we can deduce the following result.
\begin{cor}
If $\Lambda\in\text{Int}_{w}\left(PW_{\tau}^{p}\right)$, then, for
every $\epsilon>0$, $\Lambda$ is interpolating in $PW_{\tau+\epsilon}^{p}$.
\end{cor}

We also give a result involving density conditions as a second corollary
to our main result, which does not seem easy to prove directly.
\begin{cor}
Let $\Lambda$ satisfying (\ref{lambda incluse dans bande}) be a
weak interpolating sequence in $PW_{\tau}^{p}$. Then, $\mathcal{D}_{\Lambda}^{+}\leq\frac{\tau}{\pi}$.\end{cor}
\begin{proof}
It follows from Theorem \ref{thm: MAIN RESULT} that $\Lambda$ is
interpolating in $PW_{\tau+\epsilon}^{p}$, for every $\epsilon>0$.
Thus, Theorem \ref{thm:(densite et interpolation)} implies that $\mathcal{D}_{\Lambda}^{+}\leq\frac{\tau+\epsilon}{\pi}$,
for every $\epsilon>0$, thus $\mathcal{D}_{\Lambda}^{+}\leq\frac{\tau}{\pi}$.
\end{proof}

\section{Weighted Interpolation and McPhail's Condition}

The previous technics can be used to show a more general result. We
need to introduce some more definitions. Let $X$ be the Hardy space
$H^{p}\left(\mathbb{C}_{a}^{\pm}\right)$ or the Paley-Wiener space
$PW_{\tau}^{p}$, $\Lambda=\left\{ \lambda_{n}\right\} _{n\geq1}$
a sequence of complex numbers lying in the corresponding domain $\mathbb{C}_{a}^{\pm}$
or $\mathbb{C}$ and $\omega=\left(\omega_{n}\right)_{n\geq1}$ a
sequence of strictly positive numbers. We say that $\Lambda$ is $\omega-$\emph{interpolating}
in $X$ if for every $\left(a_{n}\right)_{n\geq1}\in l^{p}$, there
is $f\in X$ such that\begin{equation}
\omega_{n}f\left(\lambda_{n}\right)=a_{n},\qquad n\geq1.\label{def: weighted interpolation}\end{equation}
The reader has noticed that the previous definition of interpolation
in $X$ is equivalent to $\omega-$interpolation in $X$, with \[
\omega_{n}=\left\Vert k_{\lambda_{n}}\right\Vert _{X^{\star}}^{-1},\qquad n\geq1.\]
Let $\Lambda\subset\mathbb{C}_{a}^{\pm}$. In this section, the sequence
$\Lambda$ is \emph{a priori} not a Carleson sequence. We only assume
the Blaschke condition\begin{equation}
\sum_{n\geq1}\frac{\left|\text{Im}\left(\lambda_{n}\right)-a\right|}{1+\left|\lambda_{n}\right|^{2}}<\infty.\label{eq:Blaschke condition}\end{equation}
We set\[
\vartheta_{n}:=\prod_{k\neq n}\left|\frac{\lambda_{n}-\lambda_{k}}{\lambda_{n}-\overline{\lambda}_{k}-2ia}\right|,\qquad n\geq1.\]
The couple $\left(\Lambda,\omega\right)$ is said to satisfy the \emph{McPhail
condition} $\left(M_{q}\right)$, denoted $\left(\Lambda,\omega\right)\in\left(M_{q}\right)$,
if the measure\begin{equation}
\nu_{\Lambda,\omega}:=\sum_{n\geq1}\frac{\left|\text{Im}\left(\lambda_{n}\right)-a\right|^{q}}{\omega_{n}^{q}\vartheta_{n}^{q}}\delta_{\lambda_{n}}\label{eq:mc phail measure def}\end{equation}
is a Carleson measure in $\mathbb{C}_{a}^{\pm}$. The following theorem
is a special case of McPhail's theorem (\cite{McP90}) and is stated
as follows in \cite{JP06}.
\begin{thm}
\label{thm:(McPhail)}(McPhail)

Let $1<p<\infty$, $\Lambda\subset\mathbb{C}_{a}^{\pm}$ a sequence
satisfying the Blaschke condition (\ref{eq:Blaschke condition}) and
$\omega=\left(\omega_{n}\right)_{n\geq1}$ be a sequence of positive
numbers. The following assertions are equivalents.

$(i)$ $\Lambda$ is $\omega-$interpolating in $H^{p}\left(\mathbb{C}_{a}^{\pm}\right)$;

$(ii)$ $\left(\Lambda,\omega\right)$ satisfies the McPhail condition
$\left(M_{q}\right)$, $\frac{1}{p}+\frac{1}{q}=1$.\end{thm}
\begin{rem}
\label{Rem: weighted int PW entraine MP}It follows directly from
McPhail's Theorem and the Plancherel-Polya inequality that
if $\Lambda\subset\mathbb{C}$ is $\omega-$interpolating in $PW_{\tau}^{p}$,
then for every $a\in\mathbb{R}$, we necessarily have\[
\left(\left(\Lambda\cap\mathbb{C}_{a}^{\pm}\right),e^{\pm\tau\left|\text{Im}\left(\lambda\right)\right|}\omega\right)\in\left(M_{q}\right).\]

\end{rem}
The following result is a weighted version of Theorem \ref{thm: MAIN RESULT}.
We will only sketch the proof which is analogous to that of our main
result.
\begin{thm}
\label{thm:WEIGHTED VERSION}Let $\tau>0$, $1<p<\infty$, $\omega=\left(\omega_{n}\right)_{n\geq1}$
a sequence of strictly positive numbers and $\Lambda$ be a minimal
sequence in $PW_{\tau}^{p}$ such that for every $a\in\mathbb{R}$,
\[
\left(\left(\Lambda\cap\mathbb{C}_{a}^{\pm}\right),e^{\pm\tau\left|\text{Im}\left(\lambda\right)\right|}\omega\right)\in\left(M_{q}\right).\]
Then, for every $\epsilon>0$, $\Lambda$ is $\omega-$interpolating
in $PW_{\tau+\epsilon}^{p}$.\end{thm}
\begin{proof}
As in the proof of the main result of this paper, we fix $\epsilon>0$
and we take a finitely supported sequence $\left(a_{n}\right)_{n\geq1}$.
We consider the solution of the weighted interpolation problem (\ref{def: weighted interpolation})
given by\[
f(z)=\sum_{n\geq1}\frac{a_{n}}{\omega_{n}}f_{n}(z)H_{\epsilon}\left(z-\lambda_{n}\right).\]
As previously, it is possible to split the sum in two parts that we
estimate separately. In order to avoid technical details, let us assume
here that $\Lambda$ lies in the half-plane $\mathbb{C}_{1}^{+}$.
As before, we set \[
G_{\lambda_{n},\epsilon}\left(z\right)=e^{i\left(\tau+\epsilon\right)z}\left(z-\lambda_{n}\right)f_{n}(z)H_{\epsilon}\left(z-\lambda_{n}\right)\in H_{+}^{\infty}.\]
If $B$ denotes the Blaschke product associated to $\Lambda$, we
again write\[
G_{\lambda_{n},\epsilon}=BG_{\lambda_{n},\epsilon}^{0}\]
with $G_{\lambda_{n},\epsilon}^{0}$ still in $H_{+}^{\infty}$. By
duality , we need to estimate \[
\sup\left\{ N\left(h\right):\quad h\in H_{+}^{q},\;\left\Vert h\right\Vert _{q}=1\right\} ,\]
where \[
N(h):=\left|\sum_{n\geq1}\frac{a_{n}}{\omega_{n}}\int_{-\infty}^{\infty}\frac{G_{\lambda_{n},\epsilon}^{0}\left(x\right)h\left(x\right)}{x-\lambda_{n}}dx\right|.\]
The Cauchy formula gives then \begin{eqnarray*}
\left|\int_{-\infty}^{\infty}\frac{G_{\lambda_{n},\epsilon}^{0}\left(x\right)h\left(x\right)}{x-\lambda_{n}}dx\right| & = & \left|G_{\lambda_{n},\epsilon}^{0}\left(\lambda_{n}\right)h\left(\lambda_{n}\right)\right|\\
 & = & \frac{\left|2\text{Im}\left(\lambda_{n}\right)\right|}{\vartheta_{n}}e^{-\left(\tau+\epsilon\right)\left|\text{Im}\left(\lambda_{n}\right)\right|}\left|h\left(\lambda_{n}\right)\right|,\end{eqnarray*}
and, applying Hlder's inequality, we finally find\[
N(h)\lesssim\left(\sum_{n\geq1}\left|a_{n}\right|^{p}\right)^{\frac{1}{p}}\left(\sum_{n\geq1}\frac{\left|\text{Im}\left(\lambda_{n}\right)\right|^{q}}{\vartheta_{n}^{q}\omega_{n}^{q}}e^{-q\left(\tau+\epsilon\right)\left|\text{Im}\left(\lambda_{n}\right)\right|}\left|h\left(\lambda_{n}\right)\right|^{q}\right)^{\frac{1}{q}}.\]
By assumption, $\nu_{\Lambda,\tilde{\omega}}$, with $\tilde{\omega}=\left(\omega_{n}e^{\pm\tau\left|\text{Im}\left(\lambda_{n}\right)\right|}\right)_{n}$
(we recall that $\nu_{\Lambda,\tilde{\omega}}$ is defined by (\ref{eq:mc phail measure def})
) is a Carleson measure in $\mathbb{C}^{+}$ and so\[
\left(\sum_{n\geq1}\frac{\left|\text{Im}\left(\lambda_{n}\right)\right|^{q}}{\vartheta_{n}^{q}\omega_{n}^{q}e^{q\tau\left|\text{Im}\left(\lambda_{n}\right)\right|}}\left|h\left(\lambda_{n}\right)\right|^{q}\right)^{\frac{1}{q}}\lesssim\left\Vert h\right\Vert _{q}=1.\]
Since \[
e^{-q\epsilon\left|\text{Im}\left(\lambda_{n}\right)\right|}\leq1,\]
we obtain\[
\sup\left\{ N\left(h\right):\quad h\in H_{+}^{q},\;\left\Vert h\right\Vert _{q}=1\right\} \lesssim\left\Vert a\right\Vert _{l^{p}},\]
 which permits us to end the proof.\end{proof}
\begin{rem}
As we have seen in the previous section, the weak interpolation of
a sequence $\Lambda$ in $PW_{\tau}^{p}$ implies the interpolation
property on $\Lambda$ in $PW_{\tau+\epsilon}^{p}$, which follows
from the fact that a uniformly minimal sequence $\Lambda$ in the
Hardy space is an interpolating sequence in the same space. We wonder
if we could have an analog result in the weighted case. More precisely,
we say that the sequence $\Lambda\subset\mathbb{C}_{a}^{\pm}$ is
\emph{uniformly} $\omega-$\emph{minimal} in $H^{p}\left(\mathbb{C}_{a}^{\pm}\right)$
if there exists a sequence $\left(f_{n}\right)_{n\geq1}$ of functions
of $H^{p}\left(\mathbb{C}_{a}^{\pm}\right)$ such that\[
\omega_{n}f_{n}\left(\lambda_{k}\right)=\delta_{nk}\]
and\[
\sup_{n\geq1}\left\Vert f_{n}\right\Vert <\infty.\]
The question is to know whether a uniformly $\omega-$minimal sequence
$\Lambda$ in $H^{p}\left(\mathbb{C}_{a}^{\pm}\right)$ is necessarily
such that the couple $\left(\Lambda,\omega\right)$ satisfies the
McPhail condition $\left(M_{q}\right)$, $\frac{1}{p}+\frac{1}{q}=1$.
\end{rem}

\section{(Weak) Controllability of Linear Differential Systems}

We consider linear differential systems of the form \begin{equation}
\begin{cases}
x'(t) & =Ax(t)+Bu(t),\qquad t\geq0,\\
x(0) & =x_{0},\end{cases}\label{system AB}\end{equation}
where $A$ is the generator of a $c_{0}-$semigroup $\left(S(t)\right)_{t\geq0}$
on a Hilbert space $H$ and $B:\mathbb{C}\to H$ is an operator, called
the \emph{control operator} which we \emph{a priori} do not assume
bounded. We are thus interested in \emph{rank-$1$ control}. We refer
to \cite[Part D]{JP06} and references therein for more details on
this terminology and on the subject. Note that those authors also
consider unbounded control. We will assume that the semigroup $\left(S(t)\right)_{t\geq0}$
is exponentially stable, $i.e.$ there exists $\alpha>0$ such that
we can find $M\geq1$ for which\begin{equation}
\left\Vert S(t)\right\Vert \leq Me^{-\alpha t},\quad t\geq0.\label{def: expo stable}\end{equation}

Controlling the system (\ref{system AB}) means to act on the system
by means of a suitable \emph{input function $u$}. More precisely,
starting from an initial state $x_{0}\in H$, we want the system to
attain in time $\tau>0$ the in advance given final state $x_{1}=x(\tau)$.
Here the solution $x$ of (\ref{system AB}) is given by \begin{equation}
x(t)=S(t)x_{0}+\int_{0}^{t}S\left(t-r\right)Bu(r)dr=:S(t)x_{0}+\mathcal{B}_{t}u.\label{eq:solution AB}\end{equation}
The operator $\mathcal{B}_{t}$ is called \emph{controllability operator}
and we are interested in the study of its range, the so-called space
of reachable states. More precisely, we say that the system (\ref{system AB})
is\emph{ exactly controllable in finite time} $\tau>0$ (respectively
in \emph{infinite time}) if for every $x_{0},x_{1}\in H$, there is
$u\in L^{2}(0,\tau)$ (respectively $u\in L^{2}\left(0,\infty\right))$
such that $x(0)=x_{0}$ and $x(\tau)=x_{1}$ (respectively $\lim_{t\to\infty}x(t)=x_{1}$)
or, equivalently, if $\mathcal{B}_{\tau}$ (respectively $\mathcal{B_{\infty}}:u\mapsto\int_{0}^{\infty}S(t)Bu(t)dt$)
is surjective. It is well known that a bounded compact controllability
operator (and in particular a rank one operator) can never cover the
whole space $H$ (see \cite[p. 215]{Ni02b}).

In all what follows, we will assume that the generator $A$ admits
a Riesz basis of (normalized) eigenvectors $\left(\phi_{n}\right)_{n\geq1}$:\[
A\phi_{n}=-\lambda_{n}\phi_{n},\quad n\geq1,\]
and that the sequence of eigenvalues $\Lambda:=\left\{ \lambda_{n}\right\} _{n\geq1}$
satisfies the Blaschke condition in the right half-plane\[
\sum_{n\geq1}\frac{\text{Re}\left(\lambda_{n}\right)}{1+\left|\lambda_{n}\right|^{2}}<\infty.\]
Note that by the exponential stability, $\Lambda$ indeed lies in
the right half-plane.The Riesz basis property gives the representation\begin{equation}
H=\left\{ x=\sum_{n\geq1}a_{n}\phi_{n}:\;\sum_{n\geq1}\left|a_{n}\right|^{2}<\infty\right\} .\label{representation riesz basis}\end{equation}
We denote by $\left(\psi_{n}\right)_{n\geq1}$ the biorthogonal family
to $\left(\phi_{n}\right)_{n\geq1}$ (which also forms a Riesz basis
of $H$ and satisfies $\left\Vert \psi_{n}\right\Vert \asymp\left\Vert \phi_{n}\right\Vert ^{-1}\asymp1$).
We suppose that $B$ has the following representation\[
Bv=v\left(\sum_{n\geq1}\overline{b}_{n}\phi_{n}\right),\qquad v\in\mathbb{C},\]
with a sequence $\left(b_{n}\right)_{n\geq1}\subset\mathbb{C}$. Observe
that $B$ does not map $\mathbb{C}$ boundedly in $H$, but it does
map boundedly into some extrapolation space in which the sequence
$\left(\phi_{n}\right)_{n\geq1}$ has dense linear span: for example,
we may define \[
H_{B}:=\left\{ x:=\sum_{n\geq1}x_{n}\phi_{n}:\;\left\Vert x\right\Vert _{B}^{2}:=\sum_{n\geq1}\frac{\left|x_{n}\right|^{2}}{n^{2}\left(1+\left|b_{n}\right|^{2}\right)}<\infty\right\} .\]
 It appears that (\ref{eq:solution AB}) can be written

\begin{eqnarray*}
x(\tau) & = & S(\tau)x_{0}+\mathcal{B}_{\tau}u\\
 & = & S(\tau)x_{0}+\sum_{n\geq1}\left(\overline{b}_{n}\int_{0}^{\tau}u\left(t\right)e^{-\lambda_{n}\left(\tau-t\right)}dt\right)\phi_{n}\\
 & = & S(\tau)x_{0}+\sum_{n\geq1}\left(\overline{b}_{n}e^{-\frac{\tau}{2}\lambda_{n}}\int_{-\frac{\tau}{2}}^{\frac{\tau}{2}}\tilde{u}\left(t\right)e^{\lambda_{n}t}dt\right)\phi_{n},\end{eqnarray*}
with $\tilde{u}:=u\left(\cdot+\frac{\tau}{2}\right)\in L^{2}\left(-\frac{\tau}{2},\frac{\tau}{2}\right)$.
We have already introduced the Fourier transform $\mathcal{F}$ and
we have mentioned that $\mathcal{F}L^{2}\left(-\frac{\tau}{2},\frac{\tau}{2}\right)=PW_{\frac{\tau}{2}}^{2}.$
Hence, if\[
f:=\mathcal{F}\tilde{u}=\int_{-\frac{\tau}{2}}^{\frac{\tau}{2}}\tilde{u}(t)e^{-it\cdot}dt\in PW_{\frac{\tau}{2}}^{2},\]
we have\[
\mathcal{B}_{\tau}u=\sum_{n\geq1}\left(\overline{b}_{n}e^{i\frac{\tau}{2}\left(i\lambda_{n}\right)}f\left(i\lambda_{n}\right)\right)\phi_{n}.\]
With the same method, we obtain\[
\mathcal{B}_{\infty}u=\sum_{n\neq1}\left(\overline{b}_{n}g\left(i\lambda_{n}\right)\right)\phi_{n},\]
where $g:=\mathcal{F}u$ which belongs to $H_{+}^{2}$ from well known
facts about Hardy spaces. Since exact controllability translates to
surjectivity of $\mathcal{B}_{\tau}$ or $\mathcal{B}_{\infty}$,
and by (\ref{representation riesz basis}), we can reformulate exact
controllability in terms of a weighted interpolation problem. 
\begin{thm}
\label{thm:ECO INT}The following assertions are equivalent.

$(i)$ The system (\ref{system AB}) is exactly controllable in finite
time $\tau>0$ (respectively in infinite time);

$(ii)$ The sequence $i\Lambda=\left\{ i\lambda_{n}\right\} _{n\geq1}$
is $\omega-$interpolating in $PW_{\frac{\tau}{2}}^{2}$, with\[
\omega_{n}:=e^{-\frac{\tau}{2}\text{\emph{Im}}\left(i\lambda_{n}\right)}\left|b_{n}\right|,\qquad n\geq1\]
(respectively $\left(\left|b_{n}\right|\right)_{n\geq1}-$interpolating
in $H_{+}^{2}$).\end{thm}
\begin{rem}
As a consequence, and in view of Remark \ref{Rem: weighted int PW entraine MP},
an exact controllable system (in finite time $\tau$) has necessarily
to be such that $\left(i\Lambda,\left(\left|b_{n}\right|\right)_{n}\right)$
satisfies $\left(M_{2}\right)$ in $\mathbb{C}^{+}$ and hence the
system has to be controllable in infinite time.
\end{rem}
In \cite[p. 289-290]{Ni02b}, the author introduces a weaker type
of control, called \emph{control for simple oscillations}, requiring
that the control operator maps boundedly some Hilbert space $\mathcal{U}$
into $H$. As already mentioned above, compact (and hence finite rank)
control is impossible with such hypotheses so that we have to deal
here with unbounded control operators $B$. Nevertheless, we keep
the terminology of \cite{Ni02b} in our situation.

The system (\ref{system AB}) is said \emph{controllable for simple
oscillations} \emph{in time} $\tau>0$ if it is possible to find a
sequence $\left(u_{n}\right)_{n\geq1}$ of functions in $L^{2}\left(0,\tau\right)$
such that\begin{equation}
\mathcal{B}_{\tau}u_{n}=\phi_{n},\quad n\geq1.\label{def CSO}\end{equation}

\begin{rem}
Since \[
\left\langle \mathcal{B}_{\tau}u,\:\psi_{n}\right\rangle =\overline{b}_{n}e^{i\frac{\tau}{2}\left(i\lambda_{n}\right)}f\left(i\lambda_{n}\right),\qquad n\geq1,\]
we easily see that (\ref{def CSO}) is equivalent to the minimality
of $i\Lambda$ in $PW_{\tau}^{2}$.
\end{rem}
We can now use Theorems \ref{thm:WEIGHTED VERSION}, \ref{thm:ECO INT}
and the previous remark to establish a link between control for simple
oscillations at time $\tau$ and exact control at time $\tau+\epsilon$.
More precisely, we have the following result.
\begin{thm}
Under the above hypotheses, if the system (\ref{system AB}) is exactly
controllable in infinite time (or equivalently if $\left(i\Lambda,\left(\left|b_{n}\right|\right)_{n}\right)$
satisfies $\left(M_{2}\right)$ in $\mathbb{C}^{+}$) and if it is
controllable for simple oscillations in time $\tau>0$, then the system
is exactly controllable in finite time $\tau+\epsilon$, for every
$\epsilon>0$.\end{thm}

$\\$

\textsc{\small \'Equipe d'Analyse, Institut de Math\'ematiques de Bordeaux,
Universit\'e Bordeaux 1, 351 cours de la Lib\'eration 33405 Talence C\'edex,
France. }{\small \par}

\end{document}